\documentclass[11pt, a4paper ,reqno]{amsart}% fleqn reqno
\usepackage{syntonly}
\usepackage{pstricks,subfigure,epsfig}
\usepackage{amsmath,amsfonts,amssymb,amsthm,euscript,upgreek}
\usepackage{enumerate}
\usepackage{multirow}
\usepackage{appendix}
\usepackage{times}
\usepackage{dsfont}
\newcommand{\CP}{\mathds{C}\mathrm{P}}

\newcommand{\CH}{\mathds{C}\mathrm{H}}

\newcommand{\po}{\forall}                     
                   
\newcommand{\C}{\mathds{C}}             
            
\newcommand{\de}{\partial}

\newcommand{\Ric}{\mathrm{Ric}}

\newtheorem{theor}{Theorem}

\newtheorem{lem}[theor]{Lemma}
\newtheorem{cor}[theor]{Corollary}

\newtheorem{remar}[theor]{Remark}

\DeclareMathOperator{\Kernel}{K}
\DeclareMathOperator{\N}{N}
\begin{document}
\title{Canonical metrics on Cartan--Hartogs domains}
\author[M. Zedda]{Michela Zedda}
\address{Dipartimento di Matematica e Informatica, Universit\`{a} di Cagliari,
Via Ospedale 72, 09124 Cagliari, Italy}
\email{michela.zedda@gmail.com  }
\thanks{
The author was  supported by  RAS
through a grant financed with the ``Sardinia PO FSE 2007-2013'' funds and 
provided according to the L.R. $7/2007$.}
\date{}
\subjclass[2000]{53C55; 32Q15; 32T15.} 
\keywords{K\"ahler metrics; Extremal metrics; Engli\v{s} expansion; Cartan--Hartogs domains}
\maketitle

\begin{abstract}
In this paper we address two problems concerning a family of domains $M_{\Omega}(\mu) \subset \C^n$, called \emph{Cartan--Hartogs} domains, endowed with a natural K\"ahler metric $g(\mu)$. The first one is determining when the metric $g(\mu)$ is {\em extremal} (in the sense of Calabi), while the second one studies when the coefficient $a_2$ in the Engli\v{s} expansion of Rawnsley $\varepsilon$-function associated to $g(\mu)$ is constant.
\end{abstract}

\section{Introduction}
This paper deals with the $1$-parameter family of domains $M_{\Omega}(\mu)\subset\C^n$, called Cartan--Hartogs domains, endowed with a natural K\"ahler metric $g(\mu)$ (see next section for details). 
Cartan--Hartogs domains have been considered by several authors (see \cite{roos} and references therein) and are interesting from many points of view. With the exception of  the complex hyperbolic space which is obviously homogeneous, each Cartan--Hartogs domain $(M_{\Omega}(\mu),g(\mu))$ is a noncompact, nonhomogeneous, complete K\"ahler manifold. Further, in \cite{roos} (see also  \cite{compl}) it is shown that for a particular value $\mu_0$ of $\mu$, $g(\mu_0)$ is a K\"ahler--Einstein metric. 
In the joint work with A. Loi \cite{articwall} the author of the present paper shows that 
$\left(M_{\Omega}(\mu_0), g(\mu_0)\right)$, represent the first example of complete, nonhomogeneous K\"ahler--Einstein metric which admits a holomorphic and isometric immersion $f$ into the infinite dimensional complex projective space $\CP^\infty$, i.e. $f^*g_{FS}=g(\mu)$, where $g_{FS}$ denotes the Fubini--Study metric on $\CP^\infty$. 

In this paper we study two problems. In the first one we investigate for what choices of $\mu$ and of the basis domain $\Omega$, the metric $g(\mu)$ on the Cartan--Hartogs domain $M_{\Omega}(\mu)$ is {\em extremal} in the sense of E. Calabi (cfr. \cite{Calabi82}). Extremal metrics on compact K\"ahler manifolds have been largely studied by many mathematicians, in particular concerning the problem of existence and uniqueness in a given K\"ahler class (see \cite{ChenTian05},  \cite{ChenTian08}) and the relationship between the existence of extremal metrics and the stability of the corresponding polarized manifold (see e.g. \cite{Donaldson02}, \cite{Donaldson09}, \cite{Mabuchi04}, \cite{Szek07}, \cite{Tian97}, \cite{Tian02}). In \cite{BurnsBart} and \cite{Levin}  examples of manifolds which admit no extremal metrics are given.  In the noncompact case the existence and uniqueness of such metrics are far from being understood and many question are still open. In \cite{Chang} it has been shown the existence of a nontrivial (namely with nonconstant scalar curvature) extremal and complete K\"ahler metric in a complex one-dimensional manifold.
In \cite{canonical} it is shown that the only extremal metric on a strongly pseudoconvex Hartogs domain is the hyperbolic metric. The following theorem, which is the first main result of this paper, generalizes this result for Cartan--Hartogs domains:
\begin{theor}\label{extremalCH}
The metric $g(\mu)$ on a Cartan--Hartogs domain $(M_\Omega(\mu),g(\mu))$ is extremal if and only if is K\"ahler--Einstein.
\end{theor}

The second problem we deal with concerns Engli\v{s} expansion of the Rawnsley $\varepsilon$-function  (see Section \ref{englis} for details) associated to $\left(M_{\Omega}(\mu), g(\mu)\right)$. 
In \cite{englis2000} M. Engli\v{s} proves that any strongly pseudoconvex bounded domain of $\C^n$ with real analytic boundary admits an asymptotic expansion of the $\varepsilon$-function, whose smooth coefficients $a_j$, $j=0,1,\dots$, are computed in \cite{englisasymp} up to $j= 3$. 
Such expansion is equivalent to the Tian--Yau--Zelditch expansion of Kempf's distortion function $T_m(x)\sim\sum_{j=0}^\infty b_j(x) m^{n-j}$, $m=0,1,\dots$, for polarized compact K\"ahler manifolds (see Zelditch \cite{zelditch} and also \cite{arezzoloi2003}, \cite{graloi} and \cite{quant}), which plays a fundamental role in the geometric quantization and quantization by deformation of a K\"ahler manifold, important both from the physical and geometrical point of view.
In the joint work with A. Loi  \cite{balancedCH}, the author of the present paper studies when the
$\varepsilon$-function of Cartan--Hartogs domains is constant. More precisely, it is proven that the $\varepsilon$-function is constant if and only if $(M_\Omega(\mu),g(\mu))$ is holomorphically isometric to the complex hyperbolic space. 
It is natural to study metrics with coefficients $a_j$'s of Engli\v{s} expansion prescribed (cfr. \cite{lutian}). In a recent paper \cite{englishartogs} A. Loi and F. Zuddas proved that the Engli\v{s} expansion's $a_2$ coefficient of a $2$-dimensional strongly pseudoconvex Hartogs domain is constant if and only if the domain is holomorphically isometric to the complex hyperbolic space.
The second main result of this paper is the following theorem, where we prove a similar statement for Cartan--Hartogs domains:
\begin{theor}\label{coeffa2}
Let $(M_\Omega(\mu), g(\mu))$ be a Cartan--Hartogs domain. If the coefficient $a_2$ of Engli\v{s} expansion of the $\varepsilon$-function associated to  $g(\mu)$ is constant, then $(M_\Omega(\mu), g(\mu))$ is K\"ahler--Einstein.
\end{theor}
Observe that we believe the converse is not in general true. More precisely, we conjecture that {\em the coefficient $a_2$ of Engli\v{s} expansion of the $\varepsilon$-function associated to  $g(\mu)$ is constant iff $(M_\Omega(\mu), g(\mu))$ is biholomorphically isometric to the complex hyperbolic space}.\\

The paper includes three other sections. In the first of them we compute the curvature tensor, the Ricci and the scalar curvature of a Cartan--Hartogs domain $\left(M_{\Omega}(\mu), g(\mu)\right)$ in terms of those of the bounded symmetric domain $\Omega$ it is based on. In Section 3 we recall the definition of extremal metrics and prove our first result, Theorem \ref{extremalCH}. Finally, in the last one we recall the definition of $\varepsilon$-function and prove our second result, Theorem \ref{coeffa2}.\\

The author would like to thank Prof. A. Loi for his several comments and suggestions.

\section{K\"ahler geometry of Cartan--Hartogs domains}
From now on we denote with greek letters the indices ranging from $1$ to $d+1$ and with roman letters those going from $1$ to $d$. Further, given a coordinate system $(z_1,\dots, z_d,w)$ centered at a point $p\in M_\Omega(\mu)$, we identify  $z_{d+1}$ with $w$. Finally, let us denote with indices $w$, $\bar w$ the $d+1$, $\overline{d+1}$ entries of $g(\mu)$.

Let $(\Omega,g_B)$ be an irreducible bounded symmetric domain of $\C^d$ of genus $\gamma$, endowed with its Bergman metric $g_B$, i.e. the metric whose associated K\"ahler form is $\omega_B=-\frac i2\de\bar\de\log\N^\gamma$. Here $\N=\N(z,z)$ denotes the  {\em generic norm} of $\Omega$ defined by:
\begin{equation}%\label{genericnorm}
\N(z, z)=(V(\Omega) \Kernel(z, z))^{-\frac{1}{\gamma}},\nonumber
\end{equation}
where $V(\Omega)$ is the total volume of $\Omega$ with respect to the Euclidean measure of $\C^d$ and $\Kernel(z, z)$ is its Bergman kernel (see e.g. \cite{arazy} for more details). Recall that $g_B$ is a K\"ahler--Einstein metric, thus up to multiply by a positive constant we have:
\begin{equation}\label{detgb}
\det(g^{\Omega(\mu)})=\N^{-\gamma}.
\end{equation}
Define the family of Cartan-Hartogs domains $(M_{\Omega}(\mu),g(\mu))$, depending on the positive real number $\mu$ and based on $\Omega$, to be the domains of $\C^{d+1}$ defined by:
\begin{equation}\label{defm}
M_{\Omega}(\mu)=\left\{(z,w)\in \Omega\times\C,\ |w|^2<\N^\mu(z,z)\right\},\nonumber
\end{equation}
and endowed with the K\"ahler metric $g(\mu)$ described by the (globally defined)
K\"ahler potential centered at the origin:
\begin{equation}\label{diastM}
\Phi(z,w)=-\log(\N^\mu-|w|^2),
\end{equation}
i.e. the K\"ahler form associated to $g(\mu)$ is given by $\omega(\mu)=-\frac i2\de\bar\de \log(\N^\mu-|w|^2)$.
By (\ref{diastM}) we have: 
$$g(\mu)_{\alpha\bar\beta}=-\frac{\de^2\log(\N^\mu-|w|^2)}{\de z_\alpha\de\bar z_\beta}, \quad \po\, \alpha,\beta=1,\dots, d+1,$$
which give rise to the following matrix:
{\small\begin{equation}\label{g}
g(\mu)=\frac{1}{(\N^\mu-|w|^2)^2}
\left(
\begin{array}{ccc|c}
  &   &   &   \\
  &  (\N^\mu)_j(\N^\mu)_{\bar k}-(\N^\mu)_{j\bar k}(\N^\mu-|w|^2) & &-(\N^\mu)_{j}\ w  \\
  &   &   & \\
  \hline
  &  -(\N^\mu)_{\bar k}\ \bar w  & &\N^\mu
\end{array}
\right),
\end{equation}}
where the upper left block is a $d\times d$ submatrix. 
Let us denote $g^{\Omega(\mu)}=\frac{\mu}{\gamma}g_B$,
i.e.
\begin{equation}
g^{\Omega(\mu)}_{j\bar k}=-\frac{\de^2\log \N^\mu}{\de z_j\de\bar z_k}=\frac{ (\N^\mu)_j(\N^\mu)_{\bar k}-(\N^\mu)_{j\bar k}\N^\mu}{\N^{2\mu}},\quad \po\ j,k=1,\dots d.\nonumber
\end{equation}
 A long but straightforward computation yields the following identity, true up to multiply by a positive constant:
\begin{equation}\label{det}
\det\left(g(\mu)\right)=\frac{\N^{\mu(d+1)-\gamma}}{(\N^\mu-|w|^2)^{d+2}}.
\end{equation}
Further, if $g^{j\bar k}_{\Omega(\mu)}$ and $g(\mu)^{j\bar k}$ denote respectively the $(j, k)$-entry of the inverse matrices of $g^{\Omega(\mu)}$ and $g(\mu)$, we have:
\begin{equation}\label{gjk}
g(\mu)^{j\bar k}=\frac{\N^\mu-|w|^2}{\N^\mu}g^{j\bar k}_{\Omega(\mu)},\quad \po\ j, k=1,\dots, d.
\end{equation}
The following lemmata give the local expression of the Ricci and the scalar curvature of $g(\mu)$.
Recall that the local components of the Ricci curvature $\Ric_{g}$ and the scalar curvature $\kappa_g$ of a K\"ahler metric $g$ on a $d+1$-dimensional complex manifold are given in local coordinates respectively by:
\begin{equation}\label{ricci}
\Ric_{\alpha\bar\beta}=-\frac{\de^2\log(\det g)}{\de z_{\alpha}\de\bar z_\beta}, \quad\po\ \alpha,\beta=1,\dots,d+1,
\end{equation}
\begin{equation}\label{scal}
\kappa_g=\sum_{\alpha,\beta=1}^{d+1}g^{\beta\bar\alpha}\Ric_{\alpha\bar \beta}.
\end{equation}
\begin{lem}
The Ricci tensor of a Cartan--Hartogs domains is given by
\begin{equation}\label{ricciCH}
\Ric_{g(\mu)}=\frac{\mu(d+1)-\gamma}{\mu}
\left(
\begin{array}{ccc|c}
  &   &   & 0  \\
  & g^{\Omega(\mu)}& &\vdots \\
  &   &   &0 \\
  \hline
 0 &\dots & 0&0
 \end{array}
\right)-(d+2)g(\mu).
\end{equation}
\end{lem}
\begin{proof}
By direct computation, it  follows by (\ref{det}) and (\ref{ricci}) that in local coordinates the Ricci tensor of $g(\mu)$ reads:
\begin{equation}
\begin{split}
\Ric_{j \bar k}=&\frac{\mu(d+1)-\gamma}{\mu}g_{j\bar k}^{\Omega(\mu)}-(d+2)g(\mu)_{j\bar k},\\
\Ric_{j\bar w}=&-(d+2)g(\mu)_{j\bar w},\\
\Ric_{w\bar k}=&-(d+2)g(\mu)_{w\bar k},\\
\Ric_{w\bar w}=&-(d+2)g(\mu)_{w\bar w}.
\end{split}\nonumber
\end{equation}
\end{proof}
\begin{lem}\label{scalcurv}
The scalar curvature of a Cartan--Hartogs domain $(M_\Omega(\mu),g(\mu))$ is given by:
\begin{equation}\label{scalCH}
\kappa_{g(\mu)}=\frac{d\left(\mu(d+1)-\gamma\right)}{\mu}\frac{\N^\mu-|w|^2}{\N^{\mu}}-\left(d+2\right)\left(d+1\right).
\end{equation}
\end{lem}
\begin{proof}
By (\ref{scal}) and (\ref{ricciCH}) we get:
\begin{equation}
\kappa_{g(\mu)}=\frac{\left(\mu(d+1)-\gamma\right)}{\mu}\sum_{j,k=1}^dg(\mu)^{j\bar k}g_{k\bar j}^{\Omega(\mu)}-\left(d+2\right)\left(d+1\right).
\end{equation}
Conclusion follows by (\ref{gjk}).
\end{proof}
\begin{cor}\label{csc}
The scalar curvature of a Cartan--Hartogs domain $(M_\Omega(\mu),g(\mu))$ is constant if and only if $g(\mu)$ is K\"ahler--Einstein, i.e. $\mu=\mu_0=\frac{d+1}\gamma$.
\end{cor}
\begin{proof}
By (\ref{scalCH}) $\kappa_{g(\mu)}$ is constant if and only if $\mu=\mu_0=\frac{d+1}{\gamma}$ and by \cite[Subsection 1.3, p.13]{roos} $g(\mu_0)$ is K\"ahler--Einstein.
\end{proof}
\begin{remar}\label{characthyp}\rm
Observe that among bounded symmetric domains the complex hyperbolic space $\CH^d$ is characterized by having $\gamma=d+1$. Furthermore, if we set $\Omega=\CH^d$, $g(\mu)$ is the hyperbolic metric $g_{hyp}:=\frac1{d+1}g_B$ iff $\mu=1$, i.e. $\left( M_{\CH^d}(1),g(1)\right)\simeq \left(\CH^{d+1},g_{hyp}\right)$. Finally, once set $\Omega=\CH^d$ and $\mu=1$ in (\ref{scalCH}), one recovers the scalar curvature of $(\CH^{d+1}, g_{hyp})$, namely $\kappa_{g_{hyp}}=-(d+1)\left(d+2\right)$.
\end{remar}
\begin{remar}\label{bsd}\rm
Being a bounded symmetric domain $(\Omega, g_B)$ a K\"ahler--Einstein manifold, in our notation we have $\Ric_{g^{\Omega(\mu)}}=\Ric_{g_B}=-g_B=-\frac{\gamma}{\mu}g^{\Omega(\mu)}$ and $\kappa_{g^{\Omega(\mu)}}=\frac\gamma\mu\kappa_{g_B}=-d\frac\gamma\mu$. Furthermore, by homogeneity of $\Omega$ follows that $|R_{g^{\Omega(\mu)}}|^2=\frac{\gamma^2}{\mu^2}|R_{g_B}|^2$ is a constant, which for $\Omega=\CH^d$ and $\mu=1$ is equal to $2d(d+1)$.
\end{remar}
We conclude this section with the following lemma needed in the proof of Theorem \ref{coeffa2}, which gives the expression of the norm of Cartan--Hartogs domains' curvature tensor evaluated at any point $(0,w)\in M_\Omega(\mu)\subset \Omega\times \C$.
\begin{lem}\label{Rat0}
The norm with respect to $g(\mu)$ of the curvature tensor $R$ of a Cartan--Hartogs domain $(M_\Omega(\mu), g(\mu))$ when evaluated at any point $(0,w)\in M_\Omega(\mu)\subset \Omega\times \C$ is given by:
\begin{equation}
\begin{split}
\left[|R|^2\right]_{z=0}=&(1-|w|^2)^2|R_{g^{\Omega(\mu)}}|^2-4|w|^2(1-|w|^2)\kappa_{g^{\Omega(\mu)}}+\\
&+2d(d+1)|w|^4+4(d+1),
\end{split}
\end{equation}
where $R_{g^{\Omega(\mu)}}$ and $\kappa_{g^{\Omega(\mu)}}$ are respectively the curvature tensor and the scalar curvature of $(\Omega,g^{\Omega(\mu)})$. 
\end{lem}
\begin{proof}
Recall that, by definition, the curvature tensor and its norm read:
\begin{equation}\label{holseccurv}
R_{\alpha\bar\beta\eta\bar\tau}=-g(\mu)_{\alpha\bar\beta\eta\bar\tau}+\sum_{\zeta,\theta=1}^{d+1}g(\mu)^{\zeta\bar\theta}g(\mu)_{\alpha\bar\zeta\eta}g(\mu)_{\theta\bar\tau\bar\beta},
\end{equation}
for all $\alpha,\beta,\eta,\tau=1,\dots, d+1$,
\begin{equation}\label{normhsc}
|R|^2=\sum_{\alpha,\beta,\eta,\theta,\zeta,\nu,\xi,\tau=1}^{d+1}\overline{g(\mu)^{\alpha\bar \zeta}}g(\mu)^{\beta\bar \nu}\overline{g(\mu)^{\eta\bar \xi}}g(\mu)^{\theta\bar \tau} R_{\alpha\bar \beta \eta\bar \theta}\overline{R_{\zeta\bar \nu \xi\bar \tau}},
\end{equation}
where we denote by $g(\mu)_{\alpha\bar\zeta\eta}$, $g(\mu)_{\alpha\bar\beta\eta\bar\tau}$ respectively, the derivatives $\de g(\mu)_{\alpha\bar\zeta}/\de z_\eta$, $\de^2 g(\mu)_{\alpha\bar\beta}/\de z_\eta\de\bar z_\tau$.

Observe that the generic norm $\N$ of a bounded symmetric domain evaluated at $(0,w)$ is a positive constant which can be chosen to be $1$.  Further, $\N$ is a polynomial which does not contain monomials in the variable $z$ or $\bar z$ alone (see e.g. \cite[Prop. 7, p. 4]{articwall}), thus in particular, the derivatives of $\N^\mu$ which are taken only in the holomorphic or only in the antiholomorphic variables, vanish when evaluated at $(0,w)$. It is then easy to see that the {\em mixed components} of the inverse of $g(\mu)$ vanishes at $(0,w)$ (cfr. (\ref{g})), namely $\left[g(\mu)^{j\bar w}\right]_{z=0}=\left[g(\mu)^{w\bar k}\right]_{z=0}=0$.
Thus, for our computation of  (\ref{holseccurv}) and (\ref{normhsc}), we do not need to consider  the terms  of the sums which present such mixed components. In particular, this allows us to write (\ref{holseccurv}) as:
\begin{equation}\label{defR}
R_{\alpha\bar\beta\eta\bar\tau}=-g(\mu)_{\alpha\bar\beta\eta\bar\tau}+\sum_{p,q=1}^{d}g^{p\bar q}g(\mu)_{\alpha\bar p\eta}g(\mu)_{q\bar\tau\bar\beta}+g(\mu)^{w\bar w}g(\mu)_{\alpha\bar w  \eta}g(\mu)_{w\bar \beta\bar \tau}.
\end{equation}
In order to give an explicit expression of $\left[|R|^2\right]_{z=0}$, we need to see how the derivatives of the metric and its inverse look like when evaluated at  $(0,w)$. First, it follows easily by (\ref{detgb}), (\ref{g}) and (\ref{gjk}) (and by the discussion above) that: 
\begin{equation}\label{inverse}
\left[g(\mu)^{w\bar w}\right]_{z=0}=(1-|w|^2)^{2},\ \left[ g(\mu)^{j\bar k}\right]_{z=0}=(1-|w|^2)\left[g^{j\bar k}_{\Omega(\mu)}\right]_{z=0}.
\end{equation}
Further the nonvanishing third order derivatives of $g(\mu)$ at $(0,w)$ are
\begin{equation}\label{gjwk}
\begin{split}
\left[g(\mu)_{j\bar w \bar k}\right]_{z=0}=&\!\left[\overline{g(\mu)_{\bar j w  k}}\right]_{z=0}\!=\!\left[-w\frac{(\N^\mu)_{j\bar k}(\N^\mu-|w|^2)-2(\N^\mu)_{\bar k}(\N^\mu)_j}{(\N^\mu-|w|^2)^3}\right]_{z=0}\\
=&-w\frac{\left[(\N^\mu)_{j\bar k}\right]_{z=0}}{(1-|w|^2)^2},
\end{split}
\end{equation}
\begin{equation}\label{gwww}
 \left[g(\mu)_{w\bar w w}\right]_{z=0}= \left[\overline{g(\mu)_{\bar w w \bar w}}\right]_{z=0}=\frac{2\bar w}{(1-|w|^2)^3},
\end{equation}
\begin{equation}\label{gjkl}
\left[g(\mu)_{j\bar k l}\right]_{z=0}=\left[\overline{g(\mu)_{\bar j k\bar l}}\right]_{z=0}=-\frac{\left[(\N^\mu)_{j\bar k l}\right]_{z=0}}{1-|w|^2}.
\end{equation}
Observe in particular that we get zero in the following cases:
\begin{equation}\label{vanterm}
\left[g(\mu)_{w\bar w k}\right]_{z=0}=\left[g(\mu)_{w\bar w \bar k}\right]_{z=0}=\left[g(\mu)_{ w \bar k w}\right]_{z=0}=\left[g(\mu)_{k\bar w\bar w}\right]_{z=0}=0.
\end{equation}
The fourth order's read:
\begin{equation}\label{gijkl}
\left[g(\mu)_{i\bar jk\bar l}\right]_{z=0}=\frac{\left[g^{\Omega(\mu)}_{i\bar jk\bar l}\right]_{z=0}}{1-|w|^2}+|w|^2\frac{\left[(\N^\mu)_{i\bar l}(\N^\mu)_{k\bar j}+(\N^\mu)_{i\bar j}(\N^\mu_{k\bar l}\right]_{z=0}}{(1-|w|^2)^2},
\end{equation}
\begin{equation}\label{gwjkl}
\left[g(\mu)_{w\bar jk\bar l}\right]_{z=0}=\left[\overline{g(\mu)_{j\bar w l\bar k}}\right]_{z=0}=-\frac{\bar w\left[(\N^\mu)_{k\bar l\bar j}\right]_{z=0}}{(1-|w|^2)^2}, 
\end{equation}
\begin{equation}\label{gwwkl}
\left[g(\mu)_{w\bar w k\bar l}\right]_{z=0}=-\frac{(1+|w|^2)\left[(\N^\mu)_{k\bar l}\right]_{z=0}}{(1-|w|^2)^3},
\end{equation}
\begin{equation}\label{gwwww}
\left[g(\mu)_{j\bar w w\bar w}\right]_{z=0}=\left[g(\mu)_{w\bar j w\bar w}\right]_{z=0}=0,\quad\left[g(\mu)_{w\bar w w\bar w}\right]_{z=0}=\frac{2+4|w|^2}{(1-|w|^2)^4}.
\end{equation}
Depending on how many times the term $w$ appears, we have five kinds of tensors (\ref{defR}). Those of type $R_{i\bar j k\bar l}$ can be written as:
\begin{equation}\label{rijkl}
\left[R_{i\bar j k\bar l}\right]_{z=0}=\frac{\left[\left(R_{g^{\Omega(\mu)}}\right)_{i\bar j k\bar l}\right]_{z=0}}{1-|w|^2}-\frac{|w|^2\left[(\N^\mu)_{i\bar l}(\N^\mu)_{ k\bar j}+(\N^\mu)_{i\bar j}(\N^\mu)_{k\bar l}\right]_{z=0}}{(1-|w|^2)^2},
\end{equation}
where we used  (\ref{gjkl}), (\ref{gjwk}), (\ref{gijkl}), applied definition (\ref{holseccurv})  to $g_{\Omega(\mu)}$ and considered the identity $$\left[g^{\Omega(\mu)}_{i \bar j \bar k}\right]_{z=0}=-\left[(\N^\mu)_{i\bar j\bar k}\right]_{z=0}.$$
By applying (\ref{gjkl}), (\ref{vanterm}) and (\ref{gwjkl}), the tensors of type $R_{w\bar j k\bar l}$ (and similarly those of type  $R_{j\bar w l\bar k}$), can be written as:
\begin{equation}
\left[R_{w\bar j k\bar l}\right]_{z=0}=\frac{\bar w\left[(\N^\mu)_{k\bar l\bar j}\right]_{z=0}}{(1-|w|^2)^2}+\frac{\bar w}{(1-|w|^2)^2}\left[\sum_{p,q=1}^{d}g^{p\bar q}_{\Omega(\mu)}(\N^\mu)_{k\bar p}(\N^\mu)_{q\bar l\bar j}\right]_{z=0}.\nonumber
\end{equation}
Furthermore, since $\left[g_{j\bar k}^{\Omega(\mu)}\right]_{z=0}=-\left[(\N^\mu)_{j\bar k}\right]_{z=0}$, it follows easily the identity:
\begin{equation}\label{somme}
\left[\sum_{p,q=1}^{d}g^{p\bar q}_{\Omega(\mu)}(\N^\mu)_{k\bar p}\right]_{z=0}=\left[-\sum_{p=1}^{d}g^{p\bar q}_{\Omega(\mu)}g^{\Omega(\mu)}_{k\bar p}\right]_{z=0}=-\sum_{q=1}^{d}\delta_{k\bar q},
\end{equation}
which implies:
$$\left[\sum_{p,q=1}^{d}g^{p\bar q}_{\Omega(\mu)}(\N^\mu)_{k\bar p}(\N^\mu)_{q\bar l\bar j}\right]_{z=0}=-\left[(\N^\mu)_{k\bar l\bar j}\right]_{z=0}.$$
Thus we have: 
\begin{equation}\label{rwjkl}
\left[R_{w\bar j k\bar l}\right]_{z=0}=\left[R_{j\bar w l\bar k}\right]_{z=0}=0.
\end{equation}
The tensors of type $R_{w\bar w k\bar l}$ are given by:
{\small\begin{equation}\label{rwwkl}
\begin{split}
\left[R_{w\bar w k\bar l}\right]_{z=0}=&\frac{(1+|w|^2)\left[(\N^\mu)_{k\bar l}\right]_{z=0}}{(1-|w|^2)^3}+\frac{|w|^2}{(1-|w|^2)^3}\left[\sum_{p,q=1}^{d}g^{p\bar q}_{\Omega(\mu)}(\N^\mu)_{k\bar p}(\N^\mu)_{q\bar l}\right]_{z=0}\\
=&\frac{\left[(\N^\mu)_{k\bar l}\right]_{z=0}}{(1-|w|^2)^3},
\end{split}
\end{equation}}
where the first equality follows by  (\ref{gjwk}),  (\ref{vanterm}) and (\ref{gwwkl}), while the second one by the identity:
$$\left[\sum_{p,q=1}^{d}g^{p\bar q}_{\Omega(\mu)}(\N^\mu)_{k\bar p}(\N^\mu)_{q\bar l}\right]_{z=0}=-\left[(\N^\mu)_{k\bar l}\right]_{z=0},$$
which is an immediate consequence of (\ref{somme}). Finally, the tensors where $w$ appear three times vanish by (\ref{vanterm}) and (\ref{gwwww}), i.e. 
\begin{equation}\label{rwwwj}
\left[R_{w\bar w w\bar l}\right]_{z=0}=\left[R_{w\bar w l\bar w}\right]_{z=0}=\,0,
\end{equation}
while the ones with only $w$ terms by (\ref{inverse}), (\ref{gwww}), (\ref{vanterm}) and (\ref{gwwww}) read:
\begin{equation}\label{rwwww}
\left[R_{w\bar w w\bar w}\right]_{z=0}=-\frac{2}{(1-|w|^2)^4}.
\end{equation}
In order to compute the norm (\ref{normhsc}) recall that, besides all terms which contain mixed components of the inverse of $g(\mu)$, we can delete from the sum the terms containing $R_{w\bar w w\bar l}$ or one of its permutation, which vanishes by (\ref{rwwwj}), as well as the terms which present derivatives only in $z$ or only in $\bar z$. Thus we get:
\begin{equation}\label{Rz0}
\begin{split}
\left[|R|^2\right]_{z=0}=&\sum_{i,j,k,l,p,q,r,s=1}^{d}\left[\overline{g(\mu)^{i\bar p}}g(\mu)^{j\bar q}\overline{g(\mu)^{k\bar r}}g(\mu)^{l\bar s} R_{i\bar j k\bar l}\overline{R_{p\bar q r\bar s}}\right]_{z=0}+\\
&+4\sum_{k,l,r,s=1}^{d}\left[\overline{g(\mu)^{w\bar w}}g(\mu)^{w\bar w}\overline{g(\mu)^{k\bar r}}g(\mu)^{l\bar s} R_{w\bar w k\bar l}\overline{R_{w\bar w r\bar s}}\right]_{z=0}+\\
&+\left[\overline{g(\mu)^{w\bar w}}g(\mu)^{w\bar w}\overline{g(\mu)^{w\bar w}}g(\mu)^{w\bar w} R_{w\bar w w\bar w}\overline{R_{w\bar w w\bar w}}\right]_{z=0},
\end{split}
\end{equation}
By (\ref{inverse}), (\ref{rijkl}) and  (\ref{somme}) one has:
\begin{equation}\label{first}
\begin{split}
\sum_{i,j,k,l,p,q,r,s=1}^{d}\left[\overline{g(\mu)^{i\bar p}}g(\mu)^{j\bar q}\right.&\left.\overline{g(\mu)^{k\bar r}}g(\mu)^{l\bar s} R_{i\bar j k\bar l}\overline{R_{p\bar q r\bar s}}\right]_{z=0}=(1-|w|^2)^2|R_{g^{\Omega(\mu)}}|^2+\\
&-4|w|^2(1-|w|^2)\kappa_{g^{\Omega(\mu)}}+2d(d+1)|w|^4.
\end{split}
\end{equation}
Furthermore, by (\ref{inverse}), (\ref{somme}) and (\ref{rwwkl}), we get:
\begin{equation}\label{second}
\begin{split}
\sum_{k,l,r,s=1}^{d}\left[\overline{g(\mu)^{w\bar w}}\right.&\left.g(\mu)^{w\bar w}\overline{g(\mu)^{k\bar r}}g(\mu)^{l\bar s} R_{w\bar w k\bar l}\overline{R_{w\bar w r\bar s}}\right]_{z=0}=\\
=&\sum_{k,l,r,s=1}^{d}\left[\overline{g_{\Omega(\mu)}^{k\bar r}}g_{\Omega(\mu)}^{l\bar s}(\N^\mu)_{k\bar l}\overline{(\N^\mu)_{r\bar s}}\right]_{z=0}=d,
\end{split}
\end{equation}
and by (\ref{inverse}) and (\ref{rwwww}), we have:
\begin{equation}\label{third}
\left[\overline{g(\mu)^{w\bar w}}g(\mu)^{w\bar w}\overline{g(\mu)^{w\bar w}}g(\mu)^{w\bar w} R_{w\bar w w\bar w}\overline{R_{w\bar w w\bar w}}\right]_{z=0}=4.
\end{equation}
Substituting (\ref{first}), (\ref{second}) and (\ref{third}) in (\ref{Rz0}) we finally get:
\begin{equation}
\begin{split}
\left[|R|^2\right]_{z=0}=&(1-|w|^2)^2|R_{g^{\Omega(\mu)}}|^2-4|w|^2(1-|w|^2)\kappa_{g^{\Omega(\mu)}}+\\
&+2d(d+1)|w|^4+4(d+1),
\end{split}\nonumber
\end{equation}
as wished.
\end{proof}

\section{Extremal metrics on cartan--Hartogs domains}
The notion of extremal metrics on compact K\"ahler manifolds has been introduced by E. Calabi \cite{Calabi82} as solutions of a variational problem involving the integral of the scalar curvature. In this sense they are a natural generalization of K\"ahler--Einstein and constant scalar curvature metrics. In the noncompact case they can be defined as those metrics whose $(1,0)$-part of the Hamiltonian vector field associated to the scalar curvature is holomorphic, notion which is equivalent for compact manifold to be extremal in the variational sense. Fixing local coordinates $(z_1,\dots, z_{d+1})$ on a neighbourhood of a point $p$ belonging to a $d+1$-dimensional complex manifold $M$ endowed with a K\"ahler metric $g$, the extremal condition is given locally by the following system of PDE's (see \cite{Calabi82}):
\begin{equation}\label{scal1}
\frac{\de}{\de\bar z_\eta}\left( \sum_{\beta=1}^{d+1}g^{\beta\bar \alpha}\frac{\de\kappa_{g}}{\de \bar z_\beta}\right)=0,
\end{equation}
for all $ \alpha,\eta=1,\dots,d+1$.

We prove now our first result, Theorem \ref{extremalCH}.
\begin{proof}[Proof of Theorem \ref{extremalCH}]
By Lemma \ref{scalcurv} and Corollary \ref{csc} it is enough to show that $g(\mu)$ is not extremal for all $\mu\neq  \frac{d+1}{\gamma}$. Assume that $\mu\neq \frac{d+1}{\gamma}$. Then by (\ref{scalCH}) the scalar curvature of $g(\mu)$ is not constant and we have:
\begin{equation}
\begin{split}\label{scalderiv}
\frac{\de\kappa_{g(\mu)}}{\de \bar z_j}=&\frac{d\,\left(\mu(d+1)-\gamma\right)}{\mu}\frac{|w|^2(\N^\mu)_{\bar j}}{\N^{2\mu}},\\
\frac{\de\kappa_{g(\mu)}}{\de \bar w}=&-\frac{d\,\left(\mu(d+1)-\gamma\right)}{\mu}\frac{w}{\N^{\mu}}.
\end{split}
\end{equation}
Observe that by (\ref{g}) we can write:
$$g(\mu)^{w\bar w}=\frac{\det\left(\left[(\N^\mu)_p(\N^\mu)_{\bar q}-(\N^\mu)_{p\bar q}(\N^\mu-|w|^2)\right]\right)}{\det(g(\mu))(\N^\mu-|w|^2)^{2d}},$$
$$g(\mu)^{j\bar w}=\frac{w\sum_{k=1}^d(-1)^{j+k}(\N^\mu)_{k}\det\left(\left[(\N^\mu)_p(\N^\mu)_{\bar q}-(\N^\mu)_{p\bar q}(\N^\mu-|w|^2)\right]_{k\bar j}\right)}{\det(g(\mu))(\N^\mu-|w|^2)^{2d}},$$
where for any matrix $A$ we denote by $A_{k\bar j}$ the matrix $A$ deprived of the $k$th row and $j$th column.
Thus, developing $\det(g(\mu))$ along the last column we get:
\begin{equation}
\begin{split}
&\det(g(\mu))=\frac{\N^\mu\det\left(\left[(\N^\mu)_p(\N^\mu)_{\bar q}-(\N^\mu)_{p\bar q}(\N^\mu-|w|^2)\right]\right)}{(\N^{\mu}-|w|^2)^{2(d+1)}}+\\
&-\frac{|w|^2\sum_{j,k=1}^d(-1)^{j+k}(\N^\mu)_{k}(\N^\mu)_{\bar j}\det\left(\left[(\N^\mu)_p(\N^\mu)_{\bar q}-(\N^\mu)_{p\bar q}(\N^\mu-|w|^2)\right]_{k\bar j}\right)}{(\N^{\mu}-|w|^2)^{2(d+1)}}\\
&\qquad\qquad=\det(g(\mu))(\N^\mu-|w|^2)^{2}\left(\N^\mu g(\mu)^{w\bar w}-\bar w\sum_{j=1}^dg(\mu)^{j\bar w}(\N^\mu)_{\bar j}\right),
\end{split}\nonumber
\end{equation}
that is, 
$$\N^\mu g(\mu)^{w\bar w}-\bar w\sum_{j=1}^dg(\mu)^{j\bar w}(\N^\mu)_{\bar j}=\frac{1}{(\N^\mu-|w|^2)^{2}}.$$
Using (\ref{scalderiv}) together with this last identity, we get:
\begin{equation}
\begin{split}
\sum_{\beta=1}^{d+1}g(\mu)^{\beta\bar w}\frac{\de\kappa_{g(\mu)}}{\de \bar z_\beta}= &\frac{d\,\left(\mu(d+1)-\gamma\right)}{\mu}\frac{w}{\N^{2\mu}}\left(\bar w\sum_{j=1}^{d}g(\mu)^{j\bar w}(\N^{\mu})_{\bar j}- \N^\mu g^{w\bar w}\right)\\
=&-\frac{d\,\left(\mu(d+1)-\gamma\right)}{\mu}\frac{w}{\N^{2\mu}\, (\N^\mu-|w|^2)^{2}},
 \end{split}\nonumber
\end{equation}
which is not constant unless $\mu=\frac{\gamma}{d+1}$. Thus by system (\ref{scal1}), $g(\mu)$ is not extremal for all $\mu\neq \frac{\gamma}{d+1}$.
\end{proof}

\section{Engli\v{s} expansion for Cartan--Hartogs domains}\label{englis}
Let $M$ be a $n$-dimensional complex manifold endowed with a K\"ahler metric $g$ and let $\varphi$ be a globally defined K\"ahler potential for $g$, i.e. $\omega=\frac{i}{2}\de\bar\de\varphi$ where $\omega$ is the K\"ahler form associated to $g$. Consider the weighted Bergman space $\mathcal{H}_\alpha$ of square integrable holomorphic functions on $(M,g)$ with respect to the measure $e^{-\alpha\varphi}\frac{\omega^n}{n!}$, i.e. $f$ belongs to $\mathcal{H}_\alpha$ iff $\int_Me^{-\alpha\varphi}|f|^2\frac{\omega^n}{n!}<\infty$. Define the $\varepsilon$-function associated to $g$ to be the function:
$$\varepsilon_{\alpha g}(x)=e^{-\alpha\varphi(x)}K_{\alpha}(x,x), \quad x\in M,$$
where $K_\alpha(x,y)$ is the reproducing kernel of $\mathcal{H}_\alpha$, i.e. $K_\alpha(x,y)=\sum_j f_j(x)\bar f_j(y)$, for an orthonormal basis $\left\{f_j\right\}$ of $\mathcal{H}_\alpha$. As suggested by the notation it is not difficult to verify that $\varepsilon_{\alpha g}$ depends only on the metric $g$ and not on the choice of the K\"ahler potential $\varphi$ (which is defined up to an addition with the real part of a holomorphic function on $M$) or on the orthonormal basis chosen. In the literature the function $\varepsilon_{\alpha g}$ was first introduced under the name of $\eta$-{\em function} by J. Rawnsley in \cite{rawnsley}, later renamed as $\theta$-{\em function} in \cite{CGR}. 
In \cite{englis2000} M. Engli\v{s} proves that if $M$ is a strongly pseudoconvex bounded domain of $\C^n$ with real analytic boundary, then admits the following asymptotic expansion of $\varepsilon_{\alpha g}$, with respect to $\alpha$:
\begin{equation}\label{englisexp}
\varepsilon_{\alpha g}(x)\sim\sum_{j=0}^\infty a_j(x)\alpha^{n-j},\quad x\in M,
\end{equation}
where $a_j$, $j=0,1,\dots$ are smooth coefficients. In \cite{englisasymp} M. Engli\v{s} also computes these coefficients for $j\leq 3$ (we omit the term $a_3$ for its expression is complicated and not needed in our approach):
\begin{equation}\label{coefficients}
\left\{
\begin{array}{l}
a_0   =   1,  \\
a_1  =   \frac{1}{2}\kappa_g,  \\
a_2=\frac{1}{3}\Delta\kappa_g+\frac{1}{24}\left(|R|^2-4|\Ric|^2+3\kappa_g^2\right).
\end{array}
\right.
\end{equation}

\vspace{0.5cm}
\noindent We are now in the position of proving our second and last result, Theorem \ref{coeffa2}.
\begin{proof}[Proof of Theorem \ref{coeffa2}]
Since $g(\mu)$ is K\"ahler--Einstein iff $\mu=\mu_0=\frac{\gamma}{d+1}$ (see \cite[Subsection 1.3, p.13]{roos}), it is enough to show that if the coefficient $a_2$ evaluated at the point $(0,w)\in M_\Omega(\mu)\subset\Omega\times \C$ is constant then $\mu=\frac{\gamma}{d+1}$. Observe first that 
 by (\ref{scalCH}) and since we can assume $\left[\N^\mu\right]_{z=0}=1$ (see the discussion at the beginning of the proof of Lemma \ref{Rat0}) one has:
\begin{equation}\label{scalat0}
\left[\kappa_{g(\mu)}\right]_{z=0}=\frac{d\left(\mu(d+1)-\gamma\right)}{\mu}(1-|w|^2)-\left(d+2\right)\left(d+1\right).
\end{equation}
Furthermore, since by definition the norm of the Ricci curvature reads:
\begin{equation}
|\Ric_{g(\mu)}|^2=\sum_{\alpha,\beta,\eta,\tau=1}^{d+1}\overline{g^{\eta\bar \tau}}g^{\alpha\bar\beta}\Ric_{\eta\bar\alpha}\overline{\Ric_{\tau\bar\beta}},\nonumber
\end{equation}
by (\ref{ricciCH}) evaluated at $(0,w)$ we get:
\begin{equation}\label{riccinorm}
\begin{split}
\left[|\Ric_{g(\mu)}|^2\right]_{z=0}=&d\left(\frac{\mu(d+1)-\gamma}{\mu}\right)^2(1-|w|^2)^2+\\
&-2d(d+2)\frac{\mu(d+1)-\gamma}{\mu}(1-|w|^2)+(d+1)(d+2)^2.
\end{split}
\end{equation}
Further, by (\ref{scalCH}), (\ref{inverse}) and (\ref{somme}) we have: 
\begin{equation}\label{delta}
\begin{split}
\left[\Delta\kappa_{g(\mu)}\right]_{z=0}=&\sum_{\alpha,\beta=1}^{d+1}g(\mu)^{\alpha\bar \beta}\frac{\de^2\kappa_{g(\mu)}}{\de z_\alpha\de\bar z_{\beta}}\\
=&-\frac{d\left(\mu(d+1)-\gamma\right)}{\mu}(1-|w|^2)\left((d-1)|w|^2+1\right).
\end{split}
\end{equation}
Thus, $\left[a_2\right]_{z=0}$ is a polynomial of second order in the variable $|w|^2$ which is constant if and only if the coefficients, let us say $c_0$ and $c_1$, of $|w|^4$ and $|w|^2$ vanish. In particular, we check the necessary condition $2c_0=-c_1$:
{\small\begin{equation}
\begin{split}
&2d(d-1)\left(d+1-\frac\gamma\mu\right)+\frac14|R_{g_B}|^2\frac{\gamma^2}{\mu^2}-d\frac\gamma\mu+\frac{d(d+1)}{2}
+\left(\frac34d^2-d\right)\left(d+1-\frac\gamma\mu\right)^2\\
&=2d^2\left(d+1-\frac\gamma\mu\right)+\frac14|R_{g_B}|^2\frac{\gamma^2}{\mu^2}-\frac d2 \frac\gamma\mu+\left(\frac34d^2-d\right)\left(d+1-\frac\gamma\mu\right)^2+\\
&\qquad\qquad-\frac34d(d+1)(d+2)\left(d+1-\frac\gamma\mu\right),
\end{split}\nonumber
\end{equation}}which is obtained recalling that $\kappa_{g^{\Omega(\mu)}}=-d\frac\gamma\mu$, $|R_{g^{\Omega(\mu)}}|^2=|R_{g_B}|^2\frac{\gamma^2}{\mu^2}$ (cfr. Remark \ref{bsd}), and using Lemma \ref{Rat0}, equations (\ref{scalat0}),  (\ref{riccinorm}) and (\ref{delta}).
It follows that:
$$\frac34d^2\left(d+1-\frac\gamma\mu\right)\left(d+3\right)=0,$$
which is satisfied only by $\frac\gamma\mu=d+1$,  i.e. $\mu=\mu_0=\frac{\gamma}{d+1}$, as wished.
\end{proof}

\end{document}